\documentclass{amsart}
\usepackage[all, cmtip]{xy}
\newtheorem{lemma}{Lemma}
\newcommand{\en}{\mathbb{N}}
\newcommand{\qu}{\mathbb{Q}}
\newcommand{\zet}{\mathbb{Z}}
\newcommand{\ce}{\mathbb{C}}
\newtheorem{theorem}{Theorem}
\newtheorem{corollary}{Corollary}
\begin{document}
\title{On Hilbert's irreducibility theorem}
\author{Abel Castillo}
\address{Department of Mathematics, Statistics, and Computer Science, University of Illinois at Chicago\\
851 S Morgan St, Chicago, IL 60607}
\email{acasti8@uic.edu}
\author{Rainer Dietmann}
\address{Department of Mathematics, Royal Holloway, University of London\\
TW20 0EX Egham, United Kingdom}
\email{Rainer.Dietmann@rhul.ac.uk}
\begin{abstract}
In this paper we obtain new quantitative forms of Hilbert's Irreducibility
Theorem. In particular, we show that if $f(X, T_1, \ldots, T_s)$ is
an irreducible polynomial with integer coefficients,
having Galois group $G$ over the function field
$\qu(T_1, \ldots, T_s)$, and $K$ is any subgroup of $G$, then there
are at most
$O_{f, \varepsilon}(H^{s-1+|G/K|^{-1}+\varepsilon})$
specialisations $\mathbf{t} \in \zet^s$
with $|\mathbf{t}| \le H$ such that the resulting polynomial
$f(X)$ has Galois group $K$ over the rationals.
\end{abstract}
\maketitle
\section{Introduction}
One of the fundamental results in Diophantine geometry
is Hilbert's Irreducibility Theorem \cite{Hilbert},
stating that
if $f(X_1, \ldots, X_r, T_1, \ldots, T_s) \in \qu[X_1, \ldots, X_r,
T_1, \ldots, T_s]$ is irreducible, then there exists a specialisation
$\mathbf{t} \in \qu^s$ such that $f(X_1, \ldots, X_r)
=f(X_1, \ldots, X_r, t_1, \ldots, t_s)$ as a rational polynomial in
$X_1, \ldots, X_r$ still is irreducible over
$\qu[X_1, \ldots, X_r]$. In fact, if $r=1$ then more
is true: Suppose that $f(X, T_1, \ldots, T_s) \in \qu[X, T_1, \ldots, T_s]$
is irreducible and of degree $n$ in $X$. Consider $f$ as a polynomial in
$X$ over the rational function field $L=\qu(T_1, \ldots, T_s)$,
having roots $\alpha_1, \ldots, \alpha_n$ in the algebraic closure
$\overline{L}$. As $f$ is irreducible, these
roots are distinct, and we can consider the Galois group $G$ of $f$
over $L$ as a subgroup of the symmetric group $S_n$. Then there
exists a specialisation $\mathbf{t} \in \qu^s$ such that
the resulting rational polynomial in $X$ still is irreducible and
has Galois group $G$ over $\qu$. In fact, if $\mathbf{t}$
is chosen in such a way that the specialised polynomial in $X$
still is of degree $n$, and separable, then its Galois group $G_{\mathbf t}$
over $\qu$ is a subgroup of $G$ (well-defined up to conjugation,
see Lemma \ref{lem:UnramifiedImpliesEmbedding}
for the construction of an embedding of
$G_{\mathbf t}$ into $G$) and it turns out that
`almost all' specialisations for $\mathbf{t}$ preserve the
Galois group, i.e. $G_{\mathbf t}=G$. In this paper we are interested in getting
precise quantitative forms of these statements,
so in the setting for $r=1$ from above, for fixed $f$ and any
subgroup $K$ of $G$ let
\begin{align*}
  N_f(H; K) = \# \{\mathbf{t} \in \zet^s:
  |\mathbf{t}| \le H & \mbox{ and the splitting field
  of $f(X, \mathbf{t})$ over $\qu$} \\
  & \mbox{has Galois group $K$}\},
\end{align*}
where we use $|\cdot|$ to denote the maximum norm of a vector. Note that
without loss of generality we can assume $f$ to have integer coefficients,
and in this arithmetic setting we are counting integer specialisations
$\mathbf{t}$ of bounded height $H$. Our first result is the following.
\begin{theorem}
\label{miami}
Let $\varepsilon>0$.
Suppose that $f(X, \mathbf{T}) \in
\zet[X, T_1, \ldots, T_s]$ is irreducible.
Let $G$ be the Galois group of
$f(X)$ over $\qu(T_1, \ldots, T_s)$, and let $K$ be a subgroup of $G$. Then
\begin{equation}
\label{herbst}
  N_f(H; K) \ll_{f, \varepsilon} H^{s-1+|G/K|^{-1}+\varepsilon},
\end{equation}
where $|G/K|$ denotes the index of $K$ in $G$.
\end{theorem}
In particular, this shows that almost all specialisations for
$\mathbf{t}$ preserve the Galois group $G$ of $f$, and specialisations
leading to small subgroups of $G$ are rare.
Our result is not the first of its kind;
Cohen (see Theorem 2.1 in \cite{C}), using the large sieve,
obtained a bound in a more general number field setting, but with exponent
$s-1/2$ for $G \ne K$ instead of $s-1+|G/K|^{-1}$.
The first two bounds in the literature that are sensitive to the size of $K$
are apparently due to the second author \cite{D},
who in the special case of the polynomial
\[
  X^n + T_1 X^{n-1} + \ldots + T_n
\]
already obtained \eqref{herbst} (see \cite{R} for very recent
improvements in this special case when in addition $n \ge 12$),
and due to Zywina \cite{Z}.
Zywina, like Cohen, works over general number fields rather than the
rational numbers, but uses the larger sieve instead of the large
sieve and obtains the same bound \eqref{herbst} for the number of
all specialisations leading to a polynomial having Galois group
contained in $K$, where $K$ is allowed to be any subset of $G$
stable under conjugation, for example a normal subgroup.
Our work makes use of recent advances on bounding the number of
points on curves instead of sieve methods, generalising the
approach from \cite{D};
note that a somewhat similar line of attack was also used in a few
previous papers (see \cite{SZ}, \cite{W}, \cite{DW})
discussing the related problem of bounding the
smallest admissible specialisation in Hilbert's irreducibility theorem.
We restrict our attention to the field
of rational numbers; the method should generalise to number fields,
provided a suitable analogue of \cite{BH} holds.
It gives sharper bounds than Cohen's and Zywina's results in most cases,
namely as soon as $K$ is any non-normal subgroup of $G$
with index exceeding 2.\\
To summarise our main result on Hilbert's
irreducibility theorem, let us keep the notation from above and
introduce the quantity
\begin{align*}
  E_f(H) = \# \{\mathbf{t} \in \zet^s: |\mathbf{t}| \le H &
  \mbox{ and the splitting field
  of $f(X, \mathbf{t})$ over $\qu$} \\
  & \mbox{has Galois group different from $G$}\}.
\end{align*}
\begin{corollary}
\label{mittwoch}
Keeping the notation and the assumptions from Theorem \ref{miami}, we have
\[
  E_f(H) \ll_{f, \varepsilon} H^{s-1+\delta_G+\varepsilon},
\]
where
\[
  \delta_G =\max\{|G/K|^{-1} : \mbox{$K$ is a proper subgroup of $G$}\}.
\]
\end{corollary}
The quantity $\delta_G$ in Corollary \ref{mittwoch} for many 
groups can be as large as $\frac{1}{2}$, for example for $G=S_n$,
but for many interesting groups it can also be pretty small:
For example, if $G=A_n$ and $n \ge 5$, then $\delta_G = \frac{1}{n}$
(see \cite{DM}, Theorem 5.2A).
Coming back to the original question of
irreducibility, still assuming $r=1$, let
\[
  R_f(H) = \# \{\mathbf{t} \in \zet^s: |\mathbf{t}| \le H \mbox{ and 
  $f(X, \mathbf{t})$ becomes reducible in $\qu[X]$}\}.
\]
\begin{corollary}
\label{regen}
Keeping the notation and assumptions from Theorem \ref{miami}, we have
\begin{equation}
\label{panik}
  R_f(H) \ll_{f, \varepsilon} H^{s-1+\gamma_G+\varepsilon},
\end{equation}
where
\begin{equation}
\label{eisbaer}
  \gamma_G = \max\{|G/K|^{-1} : \mbox{$K$ is an intransitive subgroup
  of $G$}\}.
\end{equation}
\end{corollary}
Of course always $\gamma_G \le \delta_G$, but often $\gamma_G$ is much
smaller than $\delta_G$. As an example, consider $f(X,t)=X^5-t$.
The Galois group of $f(X)=f(X, t)$ over $\qu(t)$ is the dihedral
group $D_{10}$. Clearly $\delta_G=\frac{1}{2}$, but the only proper
subgroups of $D_{10}$ have $2$ or $5$ elements. Those of order $5$
are cyclic of order $5$ and therefore transitive, whereas those of
order $2$ fix one element and thus are intransitive. Consequently,
$\gamma_G=\frac{1}{5}$.
In this example, $f(X,t)$ becomes reducible as soon as $t$ is a
fifth power, so the bound \eqref{panik} actually turns out to be
sharp here.\\
Let us also
remark that in this special case $s=1$ sometimes
more can be done, see for example the papers \cite{F} and \cite{PM}.\\
Finally, let us reconsider Hilbert's Irreducibility Theorem in its
general form applying to polynomials $F(X_1, \ldots, X_r, T_1, \ldots,
T_s)$ in $r$ variables $X_1, \ldots, X_r$. This case can be reduced
to the special case $r=1$ by \textit{Kronecker's specialisation},
see for example Chapter 9, \S3 in \cite{Lang}, or the proof of
Theorem 2.5 in \cite{C}, where it has been shown that
if $f(X_1, \ldots, X_r, T_1, \ldots, T_s) \in \zet[X_1, \ldots, X_r,
T_1, \ldots, T_s]$ is irreducible over $\qu$, then for
\begin{align*}
  J_f(H) = \# & \{\mathbf{t} \in \zet^s: |\mathbf{t}| \le H \mbox{ and}\\
  & f(X_1, \ldots, X_r, \mathbf{t}) \mbox{ becomes reducible
  in $\qu[X_1, \ldots, X_r]$}\}
\end{align*}
one has the upper bound
\begin{equation}
\label{vancouver}
  J_f(H) \ll_{f} H^{s-1/2} \log H.
\end{equation}
In fact, as in Corollary \ref{regen}, the exponent $s-1/2$ in
general is sharp as can be seen for example by considering
the polynomial
\[
  f(X_1, \ldots, X_r, T_1, \ldots, T_s) =
  (X_1+\ldots+X_r)^2-(T_1+\ldots+T_s).
\]
Like in Corollary \ref{regen}, however, in special cases one can
do better.
\begin{theorem}
\label{bodensee}
Let $f(X_1, \ldots, X_r, T_1, \ldots, T_s) \in
\zet[X_1, \ldots, X_r, T_1, \ldots, T_s]$ be irreducible, and
suppose that for some $i \in \{1, \ldots, r\}$ the monomial of
highest degree in $X_i$ is of the form $X_i^n h$, where $n \ge 1$ and $h$
depends at most on $T_1, \ldots, T_s$, but not on $X_j$ for
$j \ne i$. Moreover, let $G$ be the Galois group of the splitting
field of $f(X_i)$, considered as a polynomial over the
function field
$\qu(X_1, \ldots, X_{i-1}, X_{i+1}, T_1, \ldots, T_s)$. Then
\[
  J_f(H) \ll_{f, \varepsilon} H^{s-\gamma_G+\varepsilon},
\]
where $\gamma_G$ has been defined in \eqref{eisbaer}.
\end{theorem}
Note that polynomials $f$ not satisfying the assumptions of Theorem
\ref{bodensee} regarding the form of the highest degree monomial
in one of
the variables can be brought into that form after applying a suitable
linear transformation on the variables $X_1, \ldots, X_r$, that
does not change the property of being reducible or irreducible
over the rationals. It seems difficult, though, to control how the
relevant $G$ and thus $\gamma_G$ change in this process.\\
As we have already remarked,
our approach roughly follows \cite{D}, using
auxiliary varieties based on suitable Galois resolvents,
and bounding the number of integral points on these varieties.
More care, however, has to be taken in constructing the Galois
resolvents in section \ref{bilfinger_berger} to guarantee their
irreducibility. In section \ref{bulldozer} we use a result
from the literature, stemming itself from an application of the
determinant method, to bound the number of integral points on
curves which is enough to deal with the case $s=1$. For $s>1$
we use a fibration approach to reduce to this special case
of curves. Theorem \ref{miami} along with
Corollaries \ref{mittwoch} and \ref{regen} and Theorem
\ref{bodensee} will then
be proved in sections \ref{tim}, \ref{dogs} and \ref{tom}.

\bigskip
\noindent
\textbf{Acknowledgments:} The authors would like to thank the anonymous
referees of a previous version of this article for several useful comments.

\section{Construction of the Galois resolvents}
\label{bilfinger_berger}
We will now give a construction of Galois resolvents,
polynomials that detect containment of the Galois group of a polynomial
in a prescribed group, as given in Lemma \ref{stress}. To this end we
first need some preparations. Keeping the notation from the introduction,
we observe that
the group $G$ acts on the roots of $f(X, \mathbf T)$ by permutations, and this gives rise to an injective homomorphism
$$
\rho: G \hookrightarrow S_n.
$$
For $\mathbf t = (t_1, \ldots, t_s) \in \zet^s$, write $G_{\mathbf t}$ for the Galois group of the splitting field of $f(X,\mathbf t)$. 
To make sense out of a comparison between $G_{\mathbf t}$ and subgroups of $G$, we construct an injection of $G_{\mathbf t}$ into $G$ that is compatible with the choice of enumeration of roots. In other words, we want the following diagram to commute.
\begin{equation}\label{eqn:GaloisGroupsDiagram}
\xymatrix{
G_{\mathbf t} \ar@{^{(}->}_{\rho_{\mathbf t}}[rd] \ar@{^{(}->}^{\iota}[r]
&G \ar@{^{(}->}^{\rho}[d]
\\
&S_n}
\end{equation}
\begin{lemma} \label{lem:UnramifiedImpliesEmbedding}
Suppose that $\mathbf t \in \zet^s$ satisfies the conditions
\begin{equation}\label{eqn:ConditionsUnramified}
\operatorname{deg} f(X, \mathbf t) = n \text{   and   } \Delta(\mathbf t) \neq 0,
\end{equation}
where $\Delta(\mathbf T)$ is the discriminant of $f(X,\mathbf T)$ viewed as a polynomial in $X$.
Then there exist injective homomorphisms $\iota$ and $\rho_{\mathbf t}$ such that the diagram in (\ref{eqn:GaloisGroupsDiagram}) commutes.
\end{lemma}
\begin{proof} Consider the Dedekind domain $\zet(T_1, \cdots, T_{s-1})[T_s]$; the conditions (\ref{eqn:ConditionsUnramified}) imply that the prime $(T_s - t_s)$ is unramified in the splitting field of $f(X, \mathbf T)$. Choose a prime $\mathfrak p$ in the splitting field of $f(X, \mathbf t)$ lying above $(T_s - t_s)$, and the injection of the decomposition group of this prime into $G$ gives rise of an injection of the Galois group of $f(X,T_1, \cdots, T_{s-1}, t_s)$ over $\qu(T_1, \cdots, T_{s-1})$ (see for instance \cite[Section 1.7, pp. 20-21]{S}). Since the prime $\mathfrak p$ is unramified and the reduction $f(X, \mathbf T)$ modulo $\mathfrak p$ has degree $n$ in $X$, reduction mod $\mathfrak p$ sends $\alpha_i( \mathbf T)$ to $\alpha_i(T_1, \cdots, T_{s-1}, t_s)$.
Now suppose that $\sigma$ is an element of the Galois group of $f(X,T_1, \cdots, T_{s-1}, t_s)$ over $\qu(T_1, \cdots, T_{s-1})$, and suppose that the above injection sends $\sigma \mapsto \overline \sigma$ with
$$
\overline \sigma (\alpha_i(\mathbf T)) = \alpha_j(\mathbf T).
$$
The injection described above has the property that
$$
\sigma (\alpha_i(\mathbf T) \pmod{\mathfrak p}) = \alpha_j(\mathbf T) \pmod{\mathfrak p},
$$
So we can inject the Galois group of $f(X,T_1, \cdots, T_{s-1}, \theta_s)$ into $S_n$ by its action on the roots of $f(X, \mathbf T)$, which is precisely what we need for the diagram to commute. The proof is completed by repeating this procedure one parameter at a time.
\end{proof}

\begin{lemma}
\label{ref}
Let $n \in \en$ and $z_1, \ldots, z_n, w_1, \ldots, w_n \in \ce$.
Suppose that
\[
  z_1^k + \ldots + z_n^k = w_1^k + \ldots + w_n^k
\]
for all $k \in \{1, \ldots, n\}$. Then $\{z_1, \ldots, z_n\}
=\{w_1, \ldots, w_n\}$, i.e. the $z_i$ are a permutation of the
$w_i$ and vice versa.
\end{lemma}
\begin{proof}
This is a well known result going back to Newton.
\end{proof}
\begin{lemma}
\label{resolvent}
Let $n \in \en$, let $K$ be a subgroup of $S_n$, and let
$\alpha_1, \ldots, \alpha_n \in \ce$ be distinct.
Further, let $\{\sigma_1, \ldots, \sigma_m\}$ be a set of coset
representatives for $S_n/K$, where $m=|S_n/K|$. Then there exist
$e_1, \ldots, e_n \in \en$, $d_1, \ldots, d_{|K|} \in
\en$
and $\gamma \in \zet$
such that all complex numbers
\begin{align}
\label{sonne}
  z_i & = \sum_{k=1}^{|K|} d_k \sum_{\tau \in K}
  (\alpha_{\sigma_i(\tau(1))}+\gamma)^{ke_1}
  (\alpha_{\sigma_i(\tau(2))}+\gamma)^{ke_2} \cdots
  (\alpha_{\sigma_i(\tau(n))}+\gamma)^{ke_n}\\
  & (1 \le i \le m) \nonumber
\end{align}
are distinct.
\end{lemma}
\begin{proof}
For convenience, let us introduce the notation
\begin{align*}
  \mathbf{e} &= (e_1, \ldots, e_n), \\
  w_{i, \tau, \mathbf{e}, k, \gamma} &=
  (\alpha_{\sigma_i(\tau(1))}+\gamma)^{ke_1} \cdots
  (\alpha_{\sigma_i(\tau(n))}+\gamma)^{ke_n}, \\
  w_{i, \tau, \mathbf{e}, \gamma} &= w_{i, \tau, \mathbf{e}, 1, \gamma}.
\end{align*}
We now show that it is possible to choose $\gamma \in \zet$ and
$e_1, \ldots, e_n \in \en$ in such a way that
\begin{equation}
\label{lasagne}
  w_{i, \tau_1, \mathbf{e}, \gamma} \ne w_{j, \tau_2, \mathbf{e}, \gamma}
\end{equation}
for all $(i, \tau_1)$ and $(j, \tau_2)$ where $i \ne j$ and
$\tau_1, \tau_2 \in K$.
The condition
\[
  w_{i, \tau_1, \mathbf{e}, \gamma} = w_{j, \tau_2, \mathbf{e}, \gamma}
\]
is equivalent to
\[
  \left(
    \frac{\alpha_{\sigma_i(\tau_1(1))}+\gamma}
   {\alpha_{\sigma_j(\tau_2(1))}+\gamma}
  \right)^{e_1}
  \cdots
  \left(
    \frac{\alpha_{\sigma_i(\tau_1(n))}+\gamma}
   {\alpha_{\sigma_j(\tau_2(n))}+\gamma}
  \right)^{e_n}
  = 1,
\]
providing all denominators are different from zero.
Since $i \ne j$, at least one exponent
$e_\ell$ must be attached to a fraction of the form $\frac{\alpha_s+\gamma}
{\alpha_t+\gamma}$ where $\alpha_s \ne \alpha_t$.
Suppose that all the other exponents $e_m$ where $m \ne \ell$ are fixed.
Then we are left with an equation of the form
\begin{equation}
\label{samso}
  \left( \frac{\alpha_s+\gamma}{\alpha_t+\gamma} \right)^{e_\ell} = c
\end{equation}
for some $c \in \ce$. Now choose $\gamma \in \zet$ large enough,
in terms of a sufficiently large parameter $H$ only depending on $n$, such that
\eqref{samso} has at most one
solution $e_\ell \in \en$ with $e_\ell \le H$, for all possible choices of
$\alpha_s \ne \alpha_t$, $\ell$ and $c \in \ce$.
For this fixed $\gamma$, we have shown that for all
tuples $(i, \tau_1)$ and $(j, \tau_2)$ where $i \ne j$, we have
\[
  \#\{e_1, \ldots, e_n \in \en: e_\ell \le H \, (1 \le \ell \le n)
  \mbox{ and }
  w_{i,\tau_1,\mathbf{e},\gamma}=w_{j,\tau_2,\mathbf{e},\gamma}\}
  \ll H^{n-1}.
\]
Since there are only $O_n(1)$ many possibilities to choose
$(i, \tau_1)$ and $(j, \tau_2)$ with $i \ne j$, but there are
$\gg H^n$ vectors $\mathbf{e} \in \en^n$
where $e_\ell \le H \; (1 \le \ell \le n)$, by choosing
$H$ sufficiently large we certainly can find such an exponent vector
$\mathbf{e} \in \en^n$ for which \eqref{lasagne} holds true.
Now fix that vector $\mathbf{e} \in \en^n$ and $\gamma$, and let us write
\[
  v_{i, k} = \sum_{\tau \in K} w_{i, \tau, \mathbf{e}, k, \gamma} \quad
  (1 \le i \le m, 1 \le k \le |K|).
\]
If $i \ne j$, then there exists at least one $k \in \{1, \ldots, |K|\}$
such that $v_{i, k} \ne v_{j, k}$: By Lemma \ref{ref} the conditions
$v_{i, k}=v_{j, k} \; (1 \le k \le |K|)$ would imply
\[
  \{w_{i, \tau, \mathbf{e}, \gamma} : \tau \in K\} =
  \{w_{j, \tau, \mathbf{e}, \gamma} : \tau \in K\},
\]
contradicting \eqref{lasagne}. The complex numbers in \eqref{sonne}
are now exactly of the form
\[
  z_i = \sum_{k=1}^{|K|} d_k v_{i, k} \quad (1 \le i \le m).
\]
To make them distinct, it is enough to choose $d_1, \ldots, d_{|K|} \in \en$
in such a way that
\begin{equation}
\label{schnee}
  \sum_{k=1}^{|K|} d_k (v_{i,k}-v_{j,k}) \ne 0
\end{equation}
whenever $i \ne j$. As shown above, for $i \ne j$ there is at least one
non-zero coefficient on the left hand side of \eqref{schnee}, whence
\begin{align*}
  & \#\{d_1, \ldots, d_{|K|} \in \en^{|K|}:
  \sum_{k=1}^{|K|} d_k (v_{i,k}-v_{j,k})=0
  \mbox{ and } d_k \le H \; (1 \le k \le |K|)\} \\
  & \ll H^{|K|-1}.
\end{align*}
We can now conclude in a similar way as above: Since there are $O_n(1)$
many possibilities to choose $i$ and $j$ where $i \ne j$, but
there are $\gg H^{|K|}$ many vectors $\mathbf{d} \in \en^{|K|}$ where
$d_k \le H \; (1 \le k \le |K|)$, by choosing $H$ sufficiently large
we can find a vector $\mathbf{d} \in \en^{|K|}$ such that \eqref{schnee}
is true whenever $i \ne j$. This finishes the proof.
\end{proof}
\begin{lemma}
\label{stress}
Let $n \in \en$, and let
\[
   f(X, \mathbf{T}) = X^n + g_1(\mathbf{T}) X^{n-1} + \ldots +
   g_n(\mathbf{T})
\]
where $g_i \in \zet[T_1, \ldots, T_s] \; (1 \le i \le n)$.
Suppose that $f(X)=f(X, \mathbf{T})$,
considered as a polynomial in the ring $\qu(\mathbf{T})[X]$,
has distinct roots
$\alpha_1=\alpha_1(\mathbf{T}), \ldots,
\alpha_n=\alpha_n(\mathbf{T})$ in the algebraic
closure $\overline{\qu(\mathbf{T})}$ of $\qu(\mathbf{T})$,
and let $G$ be the Galois
group of the corresponding splitting field operating on
$\alpha_1(\mathbf{T}), \ldots, \alpha_n(\mathbf{T})$.
Moreover, let $K$ be a subgroup of $G$.
Then there exists a Galois resolvent
$\Phi_{f, K}$ with the following properties:
\begin{itemize}
  \item[(i)] $\Phi_{f, K}$ is a polynomial of the form
  \begin{equation}
  \label{application}
     \Phi_{f, K}(Z, \mathbf{T}) = Z^m + h_1(\mathbf{T}) Z^{m-1} + \ldots +
     h_m(\mathbf{T}),
  \end{equation}
  where $m=|S_n/K|$ and $h_i \in \zet[T_1, \ldots, T_s]
  \; (1 \le i \le m)$.
  \item[(ii)] If one specialises the parameters $T_1, \ldots, T_s$
  in $f(X, \mathbf{T})$ to any $s$-tuple of
  integers $\mathbf t = (t_1 , \ldots, t_s)$, then if the splitting field of the polynomial
  $f(X)=f(X, \mathbf{t})$ over $\qu$ has Galois group $K$, then
  $\Phi_{f, K}(Z)=\Phi_{f, K}(Z, \mathbf{t})$ has an integer root $z$.
  \item[(iii)] If one factorises $\Phi_{f, K}(Z, \mathbf{T})$ over
  $\qu[Z, T_1, \ldots, T_s]$
  into irreducible factors, then each factor has degree
  at least $\frac{|G|}{|K|}$ in $Z$.
\end{itemize}
\end{lemma}
\begin{proof}
Since the roots $\alpha_1(\mathbf{T}), \ldots,
\alpha_n(\mathbf{T})$ are distinct, it is
possible to specialise $T_1, \ldots, T_s$ to an $s$-tuple $\mathbf{t}$ of
complex numbers such that the complex
roots $\alpha_1=\alpha_1(\mathbf{t}), \ldots,
\alpha_n=\alpha_n(\mathbf{t})$ of
$f(X)=f(X, \mathbf{t})$ are all distinct. We are therefore in a position
to invoke Lemma \ref{resolvent}. Keeping the notation from that lemma, we find
$e_1, \ldots, e_n \in \en$ and $d_1, \ldots, d_{|K|} \in \en$,
and a $\gamma \in \zet$,
such that all the numbers $z_i$ in \eqref{sonne} are distinct.
Now since replacing the variable $X$ by $X-\gamma$ does not change
the splitting field and thus does not change the
Galois group of $f(X, \mathbf{T})$ over $\qu(\mathbf{T})$,
and also for fixed $\mathbf{t} \in \zet^s$ does not change the Galois
group of $f(X, \mathbf{t})$ over $\qu$, we can without loss of generality
assume that $\gamma=0$.
We now define $\Phi_{f, K}(Z, \mathbf{T})$ to be
\begin{equation}
\label{wortschatz}
  \Phi_{f, K}(Z, \mathbf{T}) = \prod_{i=1}^m (Z-z_i)
  = \prod_{i=1}^m
  \left( Z - \sum_{k=1}^{|K|} d_k \sum_{\tau \in K}
  \alpha_{\sigma_i(\tau(1))}^{ke_1} \cdots
  \alpha_{\sigma_i(\tau(n))}^{ke_n} \right),
\end{equation}
where $\{\sigma_1, \ldots, \sigma_m\}$ is a set of coset representatives
for $S_n/K$.\\
It is important to keep in mind that by construction the
$z_i=z_i(\mathbf{T})$ are distinct, since we can specialise
$\mathbf{T}$ in such a way to end up with distinct complex $z_i$.
Expanding the expression \eqref{wortschatz}, it becomes transparent that
$\Phi_{f, K}(Z, \mathbf{T})$
is of the form \eqref{application}, where the $h_i$
are symmetric polynomials in the $z_i$ with integer coefficients.
Any permutation of the $\alpha_i$ just permutes the $z_i$, so the
$h_i$ are symmetric polynomials in the $\alpha_i$ as well,
with integer coefficients.
Hence, by the Fundamental Theorem on symmetric functions, the $h_i$ are
integer polynomials
in the elementary symmetric polynomials in
$\alpha_1, \ldots, \alpha_n$, which in turn by Vieta's Theorem are of the
form $\pm g_i$. This shows that the $h_i$ are integer polynomials in
$T_1, \ldots, T_s$ and confirms (i).\\
For the proof of (ii) and (iii) we first note that the symmetric group
$S_n$ operates on the $z_i$ via
\[
  \varrho(z_i) = \sum_{k=1}^{|K|} d_k \sum_{\tau \in K}
  \alpha_{\varrho(\sigma_i(\tau(1)))}^{ke_1} \cdots
  \alpha_{\varrho(\sigma_i(\tau(n)))}^{ke_n}
\]
for all $\varrho \in S_n$.\\
To show (ii), fix any $\mathbf{t} \in \zet^s$ and consider
\[
  \tilde{z} = \sum_{k=1}^{|K|} d_k
  \sum_{\tau \in K} \alpha_{\tau(1)}^{ke_1} \cdots
  \alpha_{\tau(n)}^{ke_n}.
\]
Choosing the $\sigma_i$ which is in the same coset of $S_n/K$
as the identity map, one finds that $\tilde{z}$ is one of the $z_i$
occurring on the left hand side of \eqref{wortschatz}.
Now suppose that $f(X)=f(X,\mathbf{t})$ has Galois group $K$ over $\qu$.
Clearly, $\tau(\tilde{z})=\tilde{z}$ for all $\tau \in K$. This shows
that $\tilde{z} \in \qu$. Moreover, $\tilde{z}$ is a root of the monic
integer
polynomial $\Phi_{f,K}(Z, \mathbf{t})$, whence the stronger conclusion
$\tilde{z} \in \zet$ follows. This finishes the proof of (ii).\\
For the proof of (iii), it is useful to work over the function field
$\qu(\mathbf{T})$ rather than over $\qu$.
As observed above, the $z_i=z_i(\mathbf{T})$ are then
\emph{distinct} elements
of the algebraic closure $\overline{\qu(\mathbf{T})}$ of
$\qu(\mathbf{T})$.
As a consequence, we obtain
\[
  \mbox{Stab}(z_i)=\{\varrho \in S_n: \varrho(z_i)=z_i\} =
  \sigma_i K \sigma_i^{-1} \quad (1 \le i \le m),
\]
which implies that
\begin{equation}
\label{absaugen}
  \{\varrho \in S_n: \varrho(z_i)=z_j\} = \sigma_j \sigma_i^{-1}
  \mbox{Stab}(z_i) = \sigma_j K \sigma_i^{-1} \quad
  (1 \le i,j \le m).
\end{equation}
These observations are crucial for the following argument: Suppose that
$\Phi_{f, K}(Z, \mathbf{T})$ factorises over $\qu$ into two factors
$\Phi_1, \Phi_2 \in \qu[Z, T_1, \ldots, T_s]$, i.e.
\begin{equation}
\label{lipom}
  \Phi_{f, K}(Z, \mathbf{T}) = \Phi_1(Z, \mathbf{T}) \Phi_2(Z, \mathbf{T}).
\end{equation}
We can consider $\Phi_{f, K}(Z)=\Phi_{f, K}(Z, \mathbf{T})$
as a monic rational
polynomial in $Z$ over $\qu(\mathbf{T})$, and analogously for $\Phi_1(Z)=
\Phi_1(Z, \mathbf{T})$. Now \eqref{wortschatz} provides a factorisation of
$\Phi_{f, K}(Z)$ over $\overline{\qu(\mathbf{T})}$ into factors of the form
$(Z-z_i)$. Suppose that $\Phi_1$ has degree $k$ in $Z$. Then by
\eqref{wortschatz}, \eqref{lipom} and uniqueness of factorisation,
$\Phi_1$ must be of the form
\[
  \Phi_1(Z) = c \cdot \prod_{j=1}^k (Z-z_{i_j})
\]
for suitable $c \in \qu$ and distinct $i_j \in \{1, \ldots, m\}$.
As shown in \eqref{absaugen}, we have
\[
  \{\varrho \in S_n: \varrho(z_{i_1}) = z_{i_l}\}
  = \sigma_{i_l} K \sigma_{i_1}^{-1}
\]
for all $l \in \{1, \ldots, k\}$. In particular, for given $l \in
\{1, \ldots, k\}$ there are exactly
$|K|$ elements in $S_n$ that map $z_{i_1}$ to $z_{i_l}$, hence there are
at most $k|K|$ many elements in $S_n$ that map $z_{i_1}$ to any root
$z_{i_l}$ of $\Phi_1$.
Therefore, if $|G|>k|K|$, then we can find an
element $\varrho \in G$ such that
\begin{equation}
\label{parkplatz}
  \varrho(z_{i_1}) \not \in \{z_{i_1}, \ldots,  z_{i_k}\},
\end{equation}
so $ \varrho(z_{i_1})$ is no root of $\Phi_1$,
as all the $z_i$ are distinct elements of $\overline{\qu(\mathbf{T})}$. 
Now the operation of
$G$ on the $z_i$ is that of field automorphisms of the splitting field
of $\Phi_{f, K}(Z)$ over $\qu(\mathbf{T})$. Such field automorphisms fix all
elements of the ground field $\qu(\mathbf{T})$
and therefore necessarily map any root of a polynomial over
$\qu(\mathbf{T})$ to another root of that polynomial.
As $\Phi_1(Z)=\Phi_1(Z, \mathbf{T})$ has coefficients in $\qu(\mathbf{T})$,
we conclude that $\varrho(z_{i_1}) \in \{z_{i_1}, \ldots,  z_{i_k}\}$.
This contradicts \eqref{parkplatz}.
Consequently, $|G|>k|K|$ is impossible. This way we obtain the lower
bound
\[
  k \ge \frac{|G|}{|K|}
\]
for the degree in $Z$ of any factor $\Phi_1$ of $\Phi_{f, K}$. This
finishes the proof of the lemma.
\end{proof}

\section{Bounding the number of integer points on curves and hypersurfaces}
\label{bulldozer}
\begin{lemma}
\label{kapstadt}
Let $f(X)=a_0X^n+a_1X^{n-1}+\ldots+a_n \in \ce[X]$. Then all roots
$z \in \ce$ of the equation $f(z)=0$ satisfy the inequality
\[
|z| \le \frac{1}{\sqrt[n]{2}-1} \cdot \max_{1 \le k \le n}
\sqrt[k]{\left| \frac{a_{k}}{a_0 {n \choose k}} \right|}.
\]
\end{lemma}
\begin{proof}
This is Theorem 3 in \S27 of \cite{M}.
\end{proof}
\begin{lemma}
\label{athen}
Let $F \in \zet[X_1, X_2]$ be irreducible and of degree $d$.
Further, let $P_1, P_2$ be real numbers such that $P_1, P_2 \ge 1$, and let
\[
  N(F; P_1, P_2) = \#\{\mathbf{x} \in \zet^2:
  F(\mathbf{x})=0
  \mbox{ and } |x_i| \le P_i \; (1 \le i \le 2)\}.
\]
Moreover, let
\[
  T = \max \left\{ \prod_{i=1}^2 P_i^{e_i} \right\}
\]
with the maximum taken over all integer $2$-tuples $(e_1, e_2)$ for
which the corresponding monomial $X_1^{e_1} X_2^{e_2}$ occurs
in $F(X_1, X_2)$ with nonzero coefficient. Then
\begin{equation}
\label{pigs}
  N(F; P_1, P_2) \ll_{d, \epsilon}
  \max\{P_1, P_2\}^{\epsilon}
  \exp \left( \frac{\log P_1 \log P_2}{\log T} \right).
\end{equation}
\end{lemma}
\begin{proof}
This is Theorem 1 in \cite{BH}; see also Lemma 8 in \cite{D} for more
details.
\end{proof}
It is crucial for our application of Lemma \ref{athen} in proving the
following result that the bound \eqref{pigs}
only depends on the degree $d$ of $F$, but not on its coefficients.
\begin{lemma}
\label{baer_maus}
Let $F(Z; T_1, \ldots, T_s) \in \zet[Z; T_1, \ldots, T_s]$ be irreducible,
and suppose that $F$ is monic of degree $m$ in $Z$. Further, let
\begin{align*}
  M_F(H) = \#\{& \mathbf{t} \in \zet^s: |\mathbf{t}| \le H
  \text{ such that}\\
  & F(Z; t_1, \ldots, t_s)=0 \text{ has an integer root $z$}\}.
\end{align*}
Then for every $\varepsilon>0$ we have
\begin{equation}
\label{obelisk_pond}
  M_F(H) \ll_{F, \varepsilon} H^{s-1+1/m+\varepsilon}.
\end{equation}
\end{lemma}
\begin{proof}
Without loss of generality we may assume that $H \ge 2$.
Then by Lemma \ref{kapstadt}, there exists
a constant $\alpha \ge 1$, depending at most on $F$,
such that whenever $\mathbf{t} \in \ce^s$ with
$|\mathbf{t}| \le H$
and $F(z; t_1, \ldots, t_s)=0$
for some $z \in \ce$, then $|z| \le H^\alpha$.
We proceed by induction on $s$. For $s=1$, Lemma \ref{athen} gives
\begin{align*}
  M_F(H) & \le \#\{(z, t_1) \in \zet^2: |z| \le H^\alpha, |t_1| \le H,
  F(z; t_1)=0\} \\ & \ll_{F, \varepsilon} H^\varepsilon
  \exp \left(\frac{\alpha (\log H)^2}{\log T} \right).
\end{align*}
Now $F(Z; T_1)$ contains the monomial $Z^m$, whence $T \ge H^{\alpha m}$,
and we obtain
\[
  M_F(H) \ll_{F, \varepsilon} H^{1/m+\varepsilon},
\]
as claimed. Next, let us discuss the case $s>1$, assuming that
the lemma has already been proved for $s-1$.
Let us first consider
those $t_2, \ldots, t_s \in \zet$ bounded in modulus by $H$ such that
$F(Z; T_1)=F(Z; T_1, t_2, \ldots, t_s)$ still is irreducible over the
rationals, as a polynomial in $Z$ and $T_1$. Then as above the number
of permissible $z$ and $t_1$ can be bounded by
$O_{F, \varepsilon}(H^{1/m+\varepsilon})$, since the term $Z^m$ is
still present. Taking into account $O(H^{s-1})$ choices for
$t_2, \ldots, t_s$, we end up with a contribution of
$O_{F, \varepsilon}(H^{s-1+1/m+\varepsilon})$, which is compatible
with \eqref{obelisk_pond}. Next, let us discuss those $t_2, \ldots, t_s
\in \zet$ bounded in modulus by $H$, such that $F(Z; T_1)$
becomes reducible over the rationals.
As $F(Z; T_1, \ldots, T_s)$ is irreducible over the rationals,
by \eqref{vancouver}
the number of such exceptional specialisations $t_2, \ldots, t_s$ can  
be bounded by $O_{F, \varepsilon}(H^{s-3/2+\varepsilon})$.
Now if $F(Z; T_1)$ becomes reducible
over the rationals, each irreducible
factor must be at least linear in $Z$, since
$F(Z; T_1)$ is monic in $Z$. If each irreducible factor is at least
quadratic in $Z$, then by the same argument as above we get a
contribution of $O_{F, \varepsilon}(H^{1/2+\varepsilon})$ for
the number of zeros of $F(Z; T_1)$, and together with
$O_{F, \varepsilon}(H^{s-3/2+\varepsilon})$ possible choices for
$t_2, \ldots, t_s$ we again end up with a bound compatible with
\eqref{obelisk_pond}. It remains to discuss those $t_2, \ldots, t_s
\in \zet$ bounded in modulus by $H$ for which $F(Z; T_1)$ splits off
a linear factor in $Z$. Let $U$ denote the number of such
$t_2, \ldots, t_s$. We can bound $U$ by the following
`fibration argument': since $F$ is an integer polynomial, monic in $Z$,
any linear factor of $F(Z; T_1)$
can be assumed to have integer coefficients and being monic in $Z$.
Given a tuple  $(t_2, \ldots, t_s)$ counted by $U$, every choice of 
$t_1 \in \zet$ gives rise to a monic integer one-variable polynomial 
$F(Z)=F(Z; t_1, t_2, \ldots, t_s)$
of degree $m$ having an integer root $z$.
But $F(Z; T_1, \ldots, T_s)$ is irreducible over
the rationals, so by Hilbert's Irreducibility Theorem
we can choose $t_1 \in \zet$ such that the specialized
polynomial $G(Z; T_2, \ldots, T_s) = F(Z; t_1, T_2, \ldots, T_s)$ still
is irreducible over the rationals.
Then $G$ still is an integer polynomial, only depending on $F$,
and monic of degree $m$ in $z$.
Therefore our inductive assumption is applicable to $G$, yielding
\[
  M_G(H) \ll_{F, \varepsilon} H^{s-2+1/m+\varepsilon}.
\]
On the other hand, as observed above,
\[
  M_G(H) \ge U,
\]
since all
$(t_2, \ldots, t_s)$ counted by $U$, for all $t_1 \in \zet$, in particular
our special choice, give rise to a specialized $F(Z)$ having an
integer root $z$. Combining the latter two bounds, we obtain
\[
  U \ll_{F, \varepsilon} H^{s-2+1/m+\varepsilon}.
\]
Once $(t_2, \ldots, t_s)$ have been fixed, we use the trivial bound
$O(H)$ for the $t_1$'s and get a total contribution of
$O_{F, \varepsilon}(H^{s-1+1/m+\varepsilon})$,
which again is compatible with
\eqref{obelisk_pond}. This finishes the proof.
\end{proof}
\section{Proof of Theorem \ref{miami}}
\label{tim}
Let us first briefly remark that without loss of generality we may
restrict to $f(X, \mathbf{T})$ that are monic in $X$: For suppose that
$f(X, \mathbf{T}) \in \zet[X, T_1, \ldots, T_s]$ of degree $n$ in $X$
is given. Then $f(X, \mathbf{T})$ is of the form
\[
  f(X, \mathbf{T}) = g_0(\mathbf{T}) X^n +
  g_1(\mathbf{T}) X^{n-1} + \ldots + g_n(\mathbf{T})
\]
for suitable $g_0, \ldots, g_n \in \zet[T_1, \ldots, T_s]$.
As $f$ is of degree $n$ in $X$, the polynomial $g_0(\mathbf{T})$ is not
identically zero and will be zero for at most $O_f(H^{n-1})$ values of
$\mathbf{t} \in \zet^s$ when $|\mathbf{t}| \le H$,
which is
of negligible order of magnitude with respect to Theorem \ref{miami}.
Now consider the polynomial
\begin{align*}
  h(X, \mathbf{T}) & = g_0(\mathbf{T})^{n-1}
  f(X/g_0(\mathbf{T}), \mathbf{T}) \\
  & = X^n + g_1(\mathbf{T}) X^{n-1} + g_0(\mathbf{T}) g_2(\mathbf{T})
  X^{n-2} + \ldots + g_0(\mathbf{T})^{n-1} g_n(\mathbf{T})
\end{align*}
in $\zet[X, T_1, \ldots, T_s]$, which shares all relevant properties
with $f(X, \mathbf{T})$: Considered over $\qu(T_1, \ldots, T_s)$, both
$f$ and $h$ have the same splitting field and hence the same Galois group,
and for fixed $\mathbf{t} \in \zet^s$ with $g_0(\mathbf{t}) \ne 0$,
again $f$ and $h$ over $\qu$ have the same splitting field and hence
the same Galois group. In particular, as $f(X, \mathbf{T})$ is irreducible
over $\qu$, the same is true for $h(X, \mathbf{T})$.
With respect to Theorem \ref{miami}, we may therefore without loss of
generality assume that $f$ is monic in $X$, i.e. $g_0(T) \equiv 1$.

Now let $\Phi_{f, K}(Z, \mathbf{T})$ be the Galois resolvent from
Lemma \ref{stress}.
Then for given $\mathbf{t} \in \zet^s$, if the polynomial
$f(X,\mathbf{t})$ has Galois group $K$ over $\qu$,
then $\Phi_{f, K}(Z, \mathbf{t})$ has an integer root $z$.
Factoring $\Phi_{f, K}(Z, \mathbf T)$
over the rationals, each irreducible factor can be assumed to
have integer coefficients, being monic in $Z$, and having degree
at least $|G|/|K|=|G/K|$ in $Z$.
Applying Lemma \ref{baer_maus} to each such
irreducible factor of $\Phi_{f, K}(Z)$, we immediately obtain
Theorem \ref{miami}.
\section{Proof of Corollary \ref{mittwoch} and \ref{regen}}
\label{dogs}
Let $n$ be the degree of $X$ in $f(X, \mathbf{T})$,
so $f(X, \mathbf{T}) = g_0(\mathbf{T}) X^n + O(X^{n-1})$ for a
suitable $g_0(\mathbf{T}) \in \zet[T_1, \ldots, T_s]$.
The two corollaries then follow from Theorem \ref{miami}
on noting that, by Lemma \ref{lem:UnramifiedImpliesEmbedding},
if $f(X, \mathbf{t})$ for some specialisation
$\mathbf{t} \in \zet^s$ still is of degree $n$ in $X$, and separable,
then the Galois group $K$ of $f$ over $\qu$ will be a subgroup of $G$.
The exceptional $\mathbf{t} \in \zet^s$ with $|\mathbf{t}| \le H$
such that $f(X, \mathbf{t})$ has degree less than $n$ or becomes
inseparable are easily seen to be of order or magnitude $O_f(H^{s-1})$
and can therefore be neglected, as they must satisfy $g_0(\mathbf{t})=0$
or $\Delta(\mathbf{t})=0$, where $\Delta(\mathbf{t})$
is the discriminant of $f(X, \mathbf{t})$. Since $f$ was assumed
to be of degree $n$, the polynomial $g_0(\mathbf{T})$ is not
identically zero, and since $f$ was assumed to be irreducible,
$\Delta(\mathbf{T})$ cannot be identically zero, whence the bound
$O(H^{s-1})$ for those exceptional $\mathbf{t}$ immediately follows.
\section{Proof of Theorem \ref{bodensee}}
\label{tom}
We follow the `fibration approach' from the proof of Lemma
\ref{baer_maus} to reduce the problem to the special case $r=1$:
Suppose that $\mathbf{t} \in \zet^s$ is counted by $J_f(H)$.
Then for this fixed $\mathbf{t}$, the specialised polynomial
$f(X_1, \ldots, X_r)$ factorises over $\qu[X_1, \ldots, X_r]$.
Now by assumption $f$ has the monomial of highest degree in $X_i$
of the form $X_i^n h$, where $h$ is not identically zero
and depends at most on $T_1, \ldots, T_s$,
but not on $X_j$ for $j \ne i$.
As there are only $O_f(H^{s-1})$ many $\mathbf{t} \in \zet^s$
with $|\mathbf{t}| \le H$ and $h(\mathbf{t})=0$,
which is a negligible quantity with respect to Theorem \ref{bodensee},
we may without loss of generality assume that $h(\mathbf{t}) \ne 0$.
Hence the polynomial $f$ factorises in
the form $f=g_1 g_2$, where $g_1$ and $g_2$ are rational polynomials
with degree less than $n$ in $X_i$. This remains true for the
resulting one-variable polynomial $f(X_i)$ after specialising all the
$X_j$ where $j \ne i$ to any rational numbers. Hence, using
Hilbert's Irreducibility Theorem to choose integer specialisations
for the variables $x_j$ with $j \ne i$ such that $f(X_i)$ as a
polynomial over $\qu(T_1, \ldots, T_s)$ keeps its Galois group $G$,
we then find that any $\mathbf{t} \in \zet^s$ counted by $J_f(H)$
leads to a specialised $f(X_i)$ that becomes reducible. Using
Corollary \ref{regen} we therefore find that
$J_f(H) \ll_{f, \varepsilon} H^{s-\gamma_G + \varepsilon}$.

\end{document}